\documentclass{amsart}
\usepackage{amsmath, amsthm, amsfonts, amssymb}
\usepackage{mathrsfs}
\usepackage{mathtools}
\usepackage{microtype}
\usepackage[utf8]{inputenc}
\providecommand{\noopsort[1]{}}
\usepackage{bbm}
\usepackage{dsfont}
\usepackage{ifthen}
\usepackage{nicefrac}
\numberwithin{equation}{section}
\usepackage{a4}
\usepackage[active]{srcltx}
\usepackage{verbatim}
\usepackage{hyperref}

\usepackage[shortlabels]{enumitem}
\setlist{leftmargin=*}
\setlist[1]{labelindent=1.2\parindent}

\newtheorem{thm}{Theorem}[section]
\newtheorem{cor}[thm]{Corollary}
\newtheorem{prop}[thm]{Proposition}
\newtheorem{lem}[thm]{Lemma}

\theoremstyle{remark}
\newtheorem{rem}[thm]{Remark}

\newtheorem{hyp}[thm]{Hypothesis}
\newtheorem{example}[thm]{Example}

\theoremstyle{definition}

\newcommand{\coloneqq}{\mathrel{\mathop:}=}

\renewcommand{\Re}{{\rm Re}\,}

\newcommand{\eps}{\varepsilon}

\newcommand{\one}{\mathds{1}}
\newcommand{\weak}{\rightharpoonup}

\newcommand{\R}{\mathds{R}}

\newcommand{\N}{\mathds{N}}

\newcommand{\bbN}{\mathbb{N}}

\newcommand{\calF}{\mathcal{F}}

\DeclarePairedDelimiter{\norm}{\lVert}{\rVert}
\DeclarePairedDelimiter{\abs}{\lvert}{\rvert}
\DeclarePairedDelimiter{\dual}{\langle}{\rangle}
\DeclarePairedDelimiter{\set}{\lbrace}{\rbrace}

\newcommand{\cB}{\mathscr{B}}
\newcommand{\cM}{\mathscr{M}}
\newcommand{\cD}{\mathscr{D}}
\newcommand{\cL}{\mathscr{L}}
\newcommand{\cA}{\mathscr{A}}

\newcommand{\cF}{\mathscr{F}}
\newcommand{\cE}{\mathscr{E}}

\newcommand{\ue}{\mathrm{e}}

\newcommand{\ud}{\mathrm{d}}

\DeclareMathOperator{\re}{Re}

\begin{document}

\title{A monotone convergence theorem for strong Feller semigroups}

\author{Christian Budde}
\address[C. Budde]{University of the Free State, Department of Mathematics and Applied Mathematics, P.\ O.\ Box 339, 9300 Bloemfontein, South Africa}
\email{BuddeCJ@ufs.ac.za}

\author{Alexander Dobrick}
\address[A. Dobrick]{Christian-Abrechts-Universit\"at zu Kiel, Arbeitsbereich Analysis, Heinrich-Hecht-Platz 6, 24118 Kiel, Germany}
\email{dobrick@math.uni-kiel.de}

\author{Jochen Gl\"uck}
\address[J. Gl\"uck]{Bergische Universit\"at Wuppertal, Fakult\"at f\"ur Mathematik und Naturwissenschaften, Gaußstr.\ 20, 42119 Wuppertal, Germany}
\email{glueck@uni-wuppertal.de}

\author{Markus Kunze}
\address[M. Kunze]{Universit\"at Konstanz, Fachbereich Mathematik und Statistik, Fach 193, 78357 Konstanz, Germany}
\email{markus.kunze@uni-konstanz.de}


\begin{abstract}
	For an increasing sequence $(T_n)$ of one-parameter semigroups of sub Markovian kernel operators over a Polish space, we study the limit semigroup and prove sufficient conditions for it to be strongly Feller. In particular, we show that the strong Feller property carries over from the approximating semigroups to the limit semigroup if the resolvent of the latter maps $\one$ to a continuous function. 
	
	This is instrumental in the study of elliptic operators on $\R^d$ with unbounded coefficients: our abstract result enables us to assign a semigroup to such an operator and to show that the semigroup is strongly Feller under very mild regularity assumptions on the coefficients.
	
	We also provide counterexamples to demonstrate that the assumptions in our main result are close to optimal.
\end{abstract}

\maketitle

\section{Introduction}

\subsection*{Convergence of semigroups}

A central aspect in the theory of strongly continuous semigroups is that of \emph{convergence} of a sequence of semigroups to a limit semigroup. Indeed, even the Hille--Yosida Theorem (see \cite[Thms. II.3.5 and II.3.8]{en}), which provides the foundation for the entire theory, is usually proved by an approximation argument (here the famous \emph{Yosida approximants} appear). More generally, this topic is the content of the Trotter--Kato Theorem (see \cite[Thms.\ III.4.8 and III.4.9]{en}) that characterizes strong convergence of semigroups by strong convergence of the resolvents of their generators.

The situation where one replaces strong convergence with weaker notions of convergence has also received some attention in the literature. Most results concern replacing norm convergence by weak convergence in the setting of Hilbert spaces, i.e.,\ studying convergence of operators in the weak operator topology. Possibly the earliest result is due to Simon \cite{simon}, who proved strong convergence of semigroups given weak convergence of the resolvents, provided that the latter are \emph{monotone} in the sense of quadratic forms. While some extensions of Simon's result beyond the setting of monotone sequences of self-adjoint generators are known (see \cite{bte14, o95}), examples show that we cannot have a full generalization of the Trotter--Kato Theorem to this setting (\cite{es10} and \cite[Ex.\ 3.5]{cte18} based on \cite{kr09}). However, Chill and terElst \cite{cte18} could prove a Trotter--Kato Theorem for positive self-adjoint operators without imposing monotonicity assumptions.\medskip

\subsection*{Theoretical contributions}

In this article, we leave the Hilbert space setting and instead focus on semigroups on the space $B_b(E)$ of bounded and measurable functions on a Polish space $E$. 
The theory of strongly continuous semigroups is not appropriate for this setting; instead, we consider semigroups which satisfy merely a measurability assumption with respect to time, and whose operators are given by transition kernels.
These are natural assumptions for transition semigroups of Markov processes; see for instance \cite{ek} for details. 

On $B_b(E)$ we will study \emph{pointwise} convergence of a sequence $(T_n)$ of semigroup to a limit semigroup $T$; 
our main result, Theorem~\ref{t.strongfeller}, gives sufficient conditions for $T$ to have the so-called strong Feller property, which means that the semigroup operators map $B_b(E)$ into the space $C_b(\Omega)$ of bounded and continuous functions.
The sequence $(T_n)$ (or, equivalently, the sequence of the Laplace transforms of the $T_n$) will be assumed to be monotone 
-- \emph{not}  in the sense of quadratic forms, considered in the above-mentioned Hilbert space setting, but pointwise.

\subsection*{An application}

The aforementioned monotonicity property is, for instance, 
satisfied in the construction of semigroups generated by elliptic operators with unbounded drift and diffusion coefficients. 
In this case the monotonicity is a consequence of the maximum principle. 
This construction was first carried out by Metafune, Pallara and Wacker \cite{mpw02} (see also
\cite{lorenzi2017} for more examples),  
and in fact, the article \cite{mpw02} was a major motivation for the present work. 
The authors of the latter article first establish monotone convergence of the resolvents of certain approximating operators. 
In a second step, they also prove (monotone) convergence of these semigroups. 
However, this second step is carried out using interior Schauder estimates, 
which requires more restrictive regularity assumptions than are necessary for the first step 
(and as a consequence, the first step is actually not used for the second step). 
With our abstract results, in particular Theorems \ref{t.limit-semigroup} and \ref{t.strongfeller}, 
convergence of the semigroups and the strong Feller property of the limit semigroup 
follow immediately from the results about the resolvent in \cite{mpw02}, 
without additional regularity assumptions on the coefficients. 
Therefore, our approach yields convergence and the strong Feller property under considerably weaker assumptions than in \cite{mpw02}.

\subsection*{Related literature}

The idea to consider sequences of semigroups or resolvents which are pointwise (or almost everywhere) increasing as a means to construct a limit semigroup is  very classical. It was already employed by Kato \cite{k54} to prove  a generation result for semigroups on the space $\ell^1$. 
This argument was later adapted to much more general settings; see in particular \cite[Sect.\ 13]{bobrowski2016}, \cite{bl17}, \cite[Thm.\ 2.1]{tyran2020} 
and \cite[Sect.\ 6]{b21}.
These articles are set in the realm of $C_0$-semigroups and the choice of the underlying space ensures that a norm bounded increasing sequence of operators is always strongly convergent. 
In contrast, our setting is more subtle and requires a careful handling of the involved modes of convergence. 
As a consequence, time regularity of the limit semigroup is far from clear in our case; see Theorem~\ref{t.strongfeller} and Example~\ref{ex.loose-strong-feller}.

\subsection*{Organization of the article}

In Section~\ref{sec.semigroups-of-kernel-operators}, we discuss preliminaries about semigroups of kernel operators and prove a continuity result for strong Feller semigroups. We take up our main line of study in Section~\ref{sec.a-monotone-convergence-theorem}, where we analyse monotone convergence of semigroups on the space $B_b(E)$ of bounded and measurable functions. 
In Section~\ref{sec.strong-feller} we prove that the limit semigroup has the strong Feller property if (A) every semigroup $T_n$ has the strong Feller property and (B) the (pseudo-)resolvent of the limit semigroup satisfies the strong Feller property. 
In Section~\ref{sec.counterexamples} we provide counterexamples that show that neither condition~(A) nor~(B) can be omitted. 
In the concluding Section~\ref{sec.application} we apply our results to construct a semigroup generated by the aforementioned elliptic operator with unbounded coefficients in the setting of \cite{mpw02}.

\section{Preliminaries on Semigroups of kernel operators}
\label{sec.semigroups-of-kernel-operators}

Throughout the article, $E$ denotes a Polish space which we always endow with its Borel $\sigma$-algebra $\cB(E)$. We denote the space of bounded and Borel measurable functions on $E$ by $B_b(E)$, and its subspace of bounded and continuous functions by $C_b(E)$.
The space of finite (signed) measures on $E$ is denoted by $\cM_b (E)$ and $\sigma \coloneqq \sigma(B_b(E), \mathscr{M}(E))$ refers to the weak topology on $B_b(E)$ induced by $\mathscr{M}(E)$. By slight abuse of notation, we will denote the restriction of this topology to $C_b(E)$ also by $\sigma$. 
We write $\weak$ to indicate convergence with respect to $\sigma$, whereas $\to$ is used to indicate convergence with respect to the supremum norm $\|\cdot\|_\infty$. Note that $f_n \weak f$ if and only if $\sup_{n\in \N}\norm{f_n}<\infty$ and $f_n(x) \to f(x)$ for all $x\in E$.\smallskip

A \emph{kernel} on $E$ is a map $k \colon E\times \cB(E)\to \R$ such that
\begin{enumerate}
	[(i)]
	\item the map $x\mapsto k(x, A)$ is measurable for all $A\in \cB (E)$;
	\item $k(x, \cdot)$ is a signed measure for every $x\in E$;
	\item $\sup_{x\in E} |k|(x,E) <\infty$, where $|k|(x, \cdot)$ denotes the total variation of the measure $k(x, \cdot)$.
\end{enumerate}
If every measure $k(x, \cdot)$ is positive (a sub probability measure), then the kernel is called \emph{positive} (\emph{sub Markovian}).

To each kernel $k$ one associates a bounded linear operator $T$ on $B_b(E)$ defined by
\begin{equation}
	\label{eq.kernelop}
	\big[Tf\big](x) \coloneqq \int_E f(y) k(x, \ud y).
\end{equation}
An operator $T$ of this form is called a \emph{kernel operator}. 
The operator $T$ is positive if and only if $k(x,A) \ge 0$ for all $x \in E$ and all $A \in \cB(E)$;
in this case, we write $T\geq 0$ and one has $\norm{T} = \sup_{x \in E} k(x,E)$.

One can prove that a bounded linear operator on $B_b(E)$ is a kernel operator if and only if it is continuous with respect to the weak topology $\sigma$ (see e.g.\ \cite[Prop.\ 3.5]{k11}). Alternatively, a bounded linear operator $T$ on $B_b(E)$ is a kernel operator if and only if its norm adjoint $T^*$ leaves the space $\cM_b(E)$ invariant (see \cite[Prop.\ 3.1]{k11}); we write
$T' \coloneqq T^*|_{\cM_b(E)}$. By $\cL(B_b(E), \sigma)$ we denote the space of $\sigma$-continuous operators on $B_b(E)$, i.e.,\ kernel operators.
\smallskip

We now turn our attention to semigroups. A \emph{semigroup of kernel operators} is a family of operators $T = (T(t))_{t>0}$ in $\cL(B_b(E), \sigma)$ such that
\begin{enumerate}
	[(i)]
	\item one has $T(t+s) = T(t)T(s)$ for all $t,s>0$;
	\item the map $(t,x)\mapsto (T(t)f)(x)$ is measurable.
\end{enumerate}
Moreover, a semigroup of kernel of operator $(T(t))_{t > 0}$ is called \emph{bounded} if there exists a number $M \ge 0$ such that $\norm{T(t)} \le M$ for all $t > 0$; it is called \emph{positive} if $T(t)\geq 0$ for all $t>0$.
We will mainly deal with bounded semigroups throughout the article. Semigroups which merely satisfy the weaker assumption $\sup_{t \in (0,1]} \norm{T(t)} < \infty$ (and are thus only exponentially bounded) can be treated by a rescaling argument.

If $T = (T(t))_{t > 0}$ is a bounded semigroup of kernel operators, then it is not difficult to see that for every $\Re\lambda >0$ there is an operator $R(\lambda) \in \cL (B_b(E), \sigma)$ such that
\[
	\langle R(\lambda)f, \mu\rangle = \int_0^\infty e^{-\lambda t} \langle T(t)f, \mu\rangle\, \ud t.
\]
In the terms of \cite[Def.\ 5.1]{k11}, this means that $T$ is an \emph{integrable semigroup} on the norming dual pair $(B_b(E), \cM_b(E))$. 
Clearly, if the semigroup $T$ is positive, then the operator $R(\lambda)$ is positive for every real number $\lambda > 0$.
By \cite[Prop.\ 5.2]{k11}, the family $(R(\lambda))_{\Re\lambda>0}$ is a \emph{pseudo resolvent}, i.e.,\ it satisfies the \emph{resolvent identity} 
\begin{align*}
	R(\lambda) - R(\mu) = (\mu-\lambda)R(\lambda)R(\mu) \quad (\Re\lambda,\, \Re\mu >0).
\end{align*}
Therefore, we call the family $(R(\lambda))_{\Re\lambda>0}$ the pseudo-resolvent of the semigroup $T$; 
alternatively, we also sometimes call it the \emph{Laplace transform} of $T$.
In general, the operators $R(\lambda)$ are not injective and thus do not form the resolvent of a (single-valued) operator. However, there is a multi-valued operator $A$ such that $R(\lambda) = (\lambda - A)^{-1}$ for all $\Re\lambda >0$ (see \cite[Appendix~A]{haase}). This operator $A$ is sometimes called the \emph{full generator} of $T$ and it is characterized by $(f,g) \in A$ if and only if
\[
	T(t)f - f = \int_0^t T(s)g\, \ud s
\]
for all $t>0$, see \cite[Proposition~5.7]{k11}. Note that this terminology is consistent with that used by Ethier--Kurtz \cite[Section~1.5]{ek}.\smallskip

Let $(T(t))_{t > 0}$ be a bounded and positive semigroup of kernel operators. 
There are a number of additional properties that such a semigroup can have and that are of particular interest for us.
We call $(T(t))_{t > 0}$ a
\emph{$C_b$-Feller semigroup} if its satisfies $T(t)C_b(E) \subseteq C_b(E)$ for each $t>0$ and, in addition, the restriction of $T$ to $C_b(E)$ is \emph{stochastically continuous}, meaning that $T(t)f \weak f$ as $t\to 0$ for every $f\in C_b(E)$. It is easy to see that if $(T(t))_{t > 0}$ leaves $C_b(E)$ invariant, then also its Laplace transform $R(\lambda)$ leaves $C_b(E)$ invariant; Example \ref{ex.cbnotgood} shows that the converse is not true in general. If $(T(t))_{t > 0}$ is a $C_b$-Feller semigroup, then $R(\lambda)|_{C_b(E)}$ is injective (see \cite[Thm.\ 2.10]{k09}) and hence the resolvent of a unique operator $A_{C_b}$. An easy computation shows that $A_{C_b}$ is exactly the part of the full generator $A$ in $C_b(E)$, i.e.,\ $u\in D(A_{C_b})$ and $A_{C_b}u=f$ if and only if
$(u,f) \in A\cap (C_b(E)\times C_b(E))$. By \cite[Thm.\ 2.10]{k09}, we can alternatively characterize $A_{C_b}$ as $\sigma$-derivative of the map $t\mapsto T(t)f$ in 0 (where we set $T(0) = I$).\smallskip

We say that the bounded and positive semigroup $(T(t))_{t > 0}$ of kernel operators enjoys the \emph{strong Feller property} if $T(t)B_b(E) \subseteq C_b(E)$ for all $t>0$. Naturally, if $(T(t))_{t > 0}$ enjoys the strong Feller property, then $T(t)C_b(E) \subseteq C_b(E)$. However, it may happen that restriction of $(T(t))_{t > 0}$ to $C_b(E)$ is not stochastically continuous. This is the case, for example, in connection with certain non-local boundary conditions \cite{akk16, k20}. In these examples, we still have that $R(\lambda)|_{C_b(E)}$ is injective. As it turns out, this already implies certain continuity properties of the semigroup.

\begin{lem}\label{l.sfcontinuous}
	Let $(T(t))_{t>0}$ be a bounded and positive semigroup of kernel operators that enjoys the strong Feller property 
	and assume that the operators $R(\lambda)$ of its pseudo-resolvent $(R(\lambda))_{\Re\lambda >0}$ are injective when restricted to $C_b(E)$. 
	Then for every $f\in B_b(E)$ the map $(t,x) \mapsto T(t)f(x)$ is continuous on $(0,\infty)\times \Omega$.
\end{lem}

\begin{proof}
	By the semigroup law, every operator $T(t)$ is the product of two strong Feller operators and thus satisfies the \emph{ultra Feller property}, i.e.,\ if $(f_n)_{n \in \N}$ is a bounded sequence of measurable functions then $(T(t)f_n)_{n \in \N}$ has a subsequence that converges uniformly on compact subsets of $E$ to some bounded continuous function, see \cite[\S~1.5]{revuz}.

	Now consider a sequence $(t_n)$ in $(0, \infty)$ that converges to a time $t \in (0,\infty)$ and a function $f\in B_b(E)$. As a preliminary step, we prove that $T(t_n)f \to T(t)f$ uniformly on compact subsets of $E$. Note that $s\coloneqq \inf\{ t_n : n\in \N\}>0$ and that the sequence $(T(t_n-s ) f)$ is bounded. As $T(s)$ is ultra Feller, passing to a subsequence, we may and shall assume that $T(t_n)f = T(s)T(t_n-s)f$ converges uniformly on compact sets to some $g\in C_b(E)$. In particular, 
$T(t_n)f \weak g$. Using the $\sigma$-continuity of the Laplace transform $R(\lambda)$ and the fact that $R(\lambda)$ commutes with every $T(t)$, we find
	\[
		R(\lambda )g = \sigma\text{-}\lim_{\mathclap{n\to\infty}} R(\lambda)T(t_n)f = \sigma\text{-}\lim_{\mathclap{n\to\infty}}T(t_n)R(\lambda) f = T(t)R(\lambda)f = R(\lambda)T(t)f;
	\]
	For the third equality we used that $T$ is strongly continuous (with respect to $\norm{\,\cdot\,}_\infty$) on the range of $R(\lambda)$  (see \cite[Rem.\ 2.5]{k09}). 
	As $R(\lambda)$ is injective on $C_b(E)$, we must have $g=T(t)f$. Now a subsequence-subsequence argument yields that $T(t_n)f \to T(t)f$ uniformly on compact subsets of $E$.

	Finally, if $(t_n, x_n) \to (t,x)$, then by what was just done, $T(t_n)f \to T(t)f$ uniformly on the compact set $\{x_n : n\in \N\}\cup \{x\}$.
\end{proof}

\section{A monotone convergence theorem}
\label{sec.a-monotone-convergence-theorem}

The purpose of this section is to prove a convergence theorem for monotone sequences of semigroups of kernel operators.
We start by characterizing domination of semigroups of kernel operators in terms of their Laplace transforms. To that end, we need the following lemma.

\begin{lem} \label{l.separating}
	There exists a countable set $\mathscr{F} \subseteq B_b(E)$ such that a measure $\mu \in \cM_b(E)$ is positive if and only if
$\langle f, \mu \rangle \geq 0$ for all $f\in \cF$.
\end{lem}

\begin{proof}
	As $E$ is Polish, its Borel $\sigma$-algebra is countably generated; let $\cE$ be a countable generator of $\cB (E)$. We may assume without loss that $\cE$ is an algebra, otherwise replacing $\cE$ with the algebra generated by it which is again countable.
	We claim that the set
	\begin{align*}
		\cF \coloneqq \set{\one_A : A \in \cE} \subset B_b(E),
	\end{align*}
	has the desired properties.
	Indeed, let $\mu \in \cM_b(E)$ and assume that $\langle f, \mu \rangle \geq 0$ for all $f\in \cF$. 
	Then
	\[
		\cE \subseteq \{ A\in \cB(E) : \mu (A)\geq 0\} \eqqcolon \cD.
	\]
	One can readily check that $\cD$ is a monotone class, so it contains the monotone class generated by $\cE$. But by the monotone class theorem, the latter is actually $\cB (E)$, so $\cD = \cB (E)$, proving $\mu \geq 0$.
\end{proof}

\begin{thm} \label{thm:domination-semigroups}
	Let $(T_1(t))_{t \geq 0}$ and $(T_2(t))_{t \geq 0}$ be positive and bounded semigroups of kernel operators, and assume that $T_j(t)C_b(E)\subseteq C_b(E)$ for all $t>0$ and $j=1,2$. 
	We denote their Laplace transforms by $(R_1(\lambda))_{\re \lambda > 0}$ and $(R_2(\lambda))_{\re \lambda > 0}$, respectively. 
	The following are equivalent:
	\begin{enumerate}
		[\upshape (i)]
		\item $R_1(\lambda) \leq R_2(\lambda)$ for all $\lambda >0$;
		\item $T_1(t) \leq T_2(t)$ for all $t>0$.
	\end{enumerate}
\end{thm}

\begin{proof}
	(i) $\Rightarrow$ (ii). 
	For each $f \in B_b(E)_+$ and $\mu \in \cM_b(E)_+$, define $r_{f, \mu} \colon (0, \infty) \to [0, \infty)$ by setting
	\begin{align*}
		r_{f, \mu}(\lambda) = \dual{(R_2(\lambda) - R_1(\lambda)) f, \mu} = \int_0^\infty \ue^{-\lambda t} \dual{(T_2(t) - T_1(t)) f, \mu} \, \ud t
	\end{align*}
	As $(R_1(\lambda))_{\re \lambda > 0}$ and $(R_2(\lambda))_{\re \lambda > 0}$ are pseudo-resolvents, the map $r_{f, \mu}$ is infinitely differentiable with
	\begin{align*}
		\frac{\ud^n}{\ud \lambda^n} r_{f, \mu}(\lambda) = (-1)^{n} n! \dual{(R_2^{n + 1}(\lambda) - R_1^{n + 1}(\lambda)) f, \mu}
	\end{align*}
	for each integer $n \ge 0$.
	Thus, the Post--Widder inversion theorem \cite[Thm.\ 1.7.7]{abhn2011} yields the existence of a null set $N(f, \mu) \subseteq (0, \infty)$ such that
	\begin{align*}
		\dual{(T_2(t) - T_1(t)) f, \mu} = \lim _{n \to \infty} (-1)^n \frac{1}{n!}\biggl(\frac{n}{t}\biggr)^{n + 1} \frac{\ud^n}{\ud \lambda^n} r_{f, \mu} \biggl(\frac{n}{t}\biggr) \geq 0
	\end{align*}
	for all $t \in (0, \infty) \setminus N(f, \mu)$. 
	
	Now let $\cF \subseteq B_b(E)$ denote a countable set with the property stated in Lemma~\ref{l.separating} and fix $\mu \in \cM_b(E)_+$.
	Then $N(\mu) = \bigcup_{f \in \cF} N(f, \mu)$ is a null set and 
	\begin{align*}
		\dual{f, (T_2'(t) - T_1'(t)) \mu} = \dual{(T_2(t) - T_1(t)) f, \mu} \geq 0.
	\end{align*}
	for all $f \in \calF$ and all $t \not\in N(\mu)$.
	By the choice of $\cF$ this implies that $T_1'(t) \mu \leq T_2'(t) \mu$ for all $t \notin N(\mu)$.
	
	Next consider a sequence $(x_n)_{n\in \N} \subset E$ that is dense in $E$ and put $N \coloneqq \bigcup_{n \in \N} N(\delta_{x_n})$, where $\delta_x$ refers to the Dirac measure in $x$. Then $N$ is a null set, and for each continuous function $f\geq 0$ we have
	\begin{align*}
		(T_2(t) - T_1(t))f)(x_n) = \langle T_2(t)f-T_1(t)f, \delta_{x_n}\rangle \geq 0
	\end{align*}
	for all $n \in \N$ and all $t \not\in N$. 
	As $T_2(t)f-T_1(t)f$ is continuous by assumption, $T_1(t)f \leq T_2(t)f$ for all $t\not\in N$.

	Finally, consider the set
	\begin{align*}
		M \coloneqq \set{t \in (0, \infty) : T_1(t)|_{C_b(E)} \leq T_2(t)|_{C_b(E)}}.
	\end{align*}
	Then $(0, \infty) \setminus M \subseteq N$, so $(0, \infty) \setminus M$ is a null set. From the semigroup law, one easily deduces that $M + M \subseteq M$. Thus \cite[Lem.\ 3.16.5]{abhn2011} implies $M = (0, \infty)$. 
	Hence, $T_1(t)f\leq T_2(t)f$ for all $t>0$ and $f\in C_b(E)$, from which (ii) easily follows.\smallskip

	The converse implication (ii) $\Rightarrow$ (i) is clear.
\end{proof}

We now turn to the main topic of this section: convergence of monotone sequences of semigroups. 
We begin, however, with single operators instead of semigroups.
Within the ordered vector space $\cL(B_b(E), \sigma)$ we use the following notation: 
if $T$ is an element and $(T_n)$ a sequence in this space, we write $T_n \uparrow T$ to say that the sequence $(T_n)$ is increasing 
(in the sense that $T_n \le T_{n+1}$ for all $n$) and that $T$ is the supremum of this sequence.

\begin{lem} 
	\label{lemma:increasing-operator-sequence-on-kb-space}
	Consider an increasing sequence $(T_n)_{n \in \N}$ of positive operators in $\cL(B_b(E), \sigma)$ 
	which satisfy $\norm{T_n} \le M$ for a number $M \ge 0$ and all $n \in \N$. 
	Then:
	\begin{enumerate}[\upshape (i)]
		\item 
		There exists a positive operator $T \in \cL(B_b(E), \sigma)$ of norm $\norm{T} \le M$
		such that $T_n f \weak T f$ for each $f \in B_b(E)$. 
		In particular, $T_n \uparrow T$. 
		
		\item 
		If $(S_n)_{n \in \N}$ is another norm bounded increasing sequence of positive operators in $\cL(B_b(E), \sigma)$, 
		whose limit we denote by $S$, then $S_n T_n \uparrow S T$. 
	\end{enumerate}
	Corresponding statements hold for decreasing sequences of positive operators.
\end{lem}

\begin{proof}
	We only prove the results for increasing sequences, the case of decreasing sequences is similar. 

	(i) 
	We denote the kernel associated to $T_n$ by $k_n$ and put $k(x,A) \coloneqq \sup_n k_n(x,A)$. 
	It follows from $\norm{T_n} \le M$ for all $n$ that $k(x, E) \leq M$ for all $x\in E$. 
	Moreover, $k$ is again a kernel, cf.\ \cite[Lem.\ 3.5]{gk15}. 
	We denote the kernel operator associated with $k$ by $T$. 
	Then $\norm{T} \le M$ and it follows from the definition of $k$ that $T_nf \weak Tf$ for all $f\in B_b(E)$;
	hence, $T=\sup_n T_n$.
	\smallskip

	(ii) 
	According to~(i), the supremum $\sup_{n \in \N} S_n T_n$ exists, and we have to show that it is equal to $ST$. 
	Clearly, $S_n T_n \leq ST$. 
	On the other hand, for $f\in B_b(E)_+$ and $m\in \N$, we have
	\begin{align*}
		\big(\sup_{n \in \N} S_n T_n\big) f \geq \big(\sup_{n \in \N} S_n T_m\big) f = \sigma\text{-}\lim_n (S_n T_m) f = \sigma\text{-}\lim_n S_n (T_m f) = S T_m f,
	\end{align*}
	where the second and the last equality follow from~(i). 
	However, $T_mf\weak Tf$ and thus, by the $\sigma$-continuity of $S$, 
	the latter converges to $STf$, which proves the converse inequality $\sup_{n \in \N} S_n T_n \geq ST$.
\end{proof}

It is now easy to show that  monotone and uniformly bounded sequences of semigroups have limit semigroups.

\begin{thm} 
	\label{t.limit-semigroup}
	Let $(T_n(t))_{t \geq 0}$ be a sequence of bounded and positive semigroups of kernel operators with Laplace transforms $(R_n(\lambda))_{\re \lambda > 0}$, and assume that there exists a number $M \ge 0$ such that $\norm{T_n(t)} \le M$ for all $t > 0$ and all $n$.
	Then the following are equivalent:
	\begin{enumerate}[\upshape (i)]
		\item For every $t > 0$ the sequence $(T_n(t))$ is increasing (decreasing).
		\item For every $\lambda >0$ the sequence $(R_n(\lambda))$ is increasing (decreasing).
	\end{enumerate}
	If the equivalent assertions {\upshape(i)} and {\upshape(ii)} are satisfied the family of operators $(T(t))_{t \geq 0}$, defined by 
	\begin{align*}
		T(t) f := \sigma\text{-}\lim_{\mathclap{n\to\infty}} T_n(t) f
	\end{align*}
	for each $f \in B_b(E)$, is a positive and bounded semigroup of kernel operators with the property $\norm{T(t)} \le M$ for all $t > 0$. 
	Its pseudo-resolvent $(R(\lambda))_{\re \lambda > 0}$ is given by
	\begin{align*}
		R(\lambda) f  = \sigma\text{-}\lim_{\mathclap{n\to\infty}} R_n(\lambda) f.
	\end{align*}
	for all $f \in B_b(E)$ and $\re \lambda > 0$. 
\end{thm}

\begin{proof}
	The equivalence of (i) and (ii) is immediate from Theorem \ref{thm:domination-semigroups}, 
	so assume now that~(i) and~(ii) are satisfied.
	
	It follows from Lemma~\ref{lemma:increasing-operator-sequence-on-kb-space}(i) that, for each $t > 0$, the operator $T(t) \in \cL(B_b(E), \sigma)$ is positive and has norm at most $M$. 
	Lemma~\ref{lemma:increasing-operator-sequence-on-kb-space}(ii) yields the semigroup law for $T$. 
	Moreover, the function 
	\begin{align*}
		(t,x) \mapsto (T(t)f)(x) = \lim_{n\to\infty} \big(T_n(t)f\big) (x)
	\end{align*}
	is clearly measurable for all $f \in B_b(E)$. 
	Hence, $(T(t))_{t > 0}$ defines semigroup of kernel operators.
	
	As explained in Section~\ref{sec.semigroups-of-kernel-operators}, the pseudo-resolvent operators $R(\lambda) \in \cL(B_b(E), \sigma)$ satisfy
	\begin{align*}
		\dual{R(\lambda) f, \mu} \coloneqq \int_0^\infty \ue^{-\lambda t} \dual{T(t) f, \mu} \, \ud t 
	\end{align*}
	for all $f \in B_b(E)$, $\mu \in \cM_b(E)$ and $\re \lambda > 0$. 
	Since an analogous formula holds for $\dual{R_n(\lambda) f, \mu}$ for each $n$, the convergence of $\dual{R_n(\lambda) f, \mu}$ to $\dual{R(\lambda) f, \mu}$ follows from the dominated convergence theorem. 
\end{proof}

The convergence $R_n(\lambda)f \weak R(\lambda)f$ for all $f\in B_b(E)$ in Theorem~\ref{t.limit-semigroup} can equivalently be expressed in terms of the full generators:

\begin{cor}
	\label{c.generatorconv}
	Assume that, in the situation of Theorem~\ref{t.limit-semigroup}, the equivalent assertions {\upshape(i)} and {\upshape(ii)} are satisfied.
	Denote the full generator of $T$ by $A$ and, for every $n\in \N$, the full generator of $T_n$ by $A_n$ . 
	Then we have $(u,f) \in A$ if and only if there is a sequence of pairs $(u_n,f_n)\in A_n$ with $u_n\weak f$ and $f_n\weak f$.
\end{cor}

\begin{proof}
	First assume that $(u,f)\in A$, which is equivalent to $u=R(1)(u - f)$. 
	We set $g\coloneqq u-f$ and $u_n \coloneqq R_n(1)g$.
	By assumption $u_n \weak u$. Moreover, $f_n \coloneqq  u_n - g \weak u- g =f$. As $(u_n, f_n) \in A_n$, we have found a sequence as claimed.

	To see the converse, assume that $(u_n,f_n) \in A_n$ is such that $u_n \weak u$ and $f_n \weak f$. Let us set $g_n \coloneqq u_n - f_n$ and
	$g= u-f$. Then $R_n(1)g_n = u_n \weak u$. On the other hand, $R_n(1)g_n \weak R(1)g$. Indeed, in the case where $R_n(1)\uparrow R(1)$ we have
	\begin{align*}
		\abs{R_n(1)g_n - R(1)g} & \leq \abs{R_n(1)(g_n-g)} + \abs{R_n(1)g - R(1)g} \\
		& \leq R(1) \abs{g_n-g} + \abs{R_n(1)g - R(1)g} \to 0,
	\end{align*}
	as this is true for the second term by assumption and follows for the first term from the $\sigma$-continuity of $R(1)$. In the case where $R_n(1)\downarrow R(1)$ one can argue similarly. Combining these two facts, it follows that $u=R(1)g = R(1)(u-f)$, which is equivalent to $(u,f) \in A$.
\end{proof}

\section{The strong Feller property for the limit semigroup}
\label{sec.strong-feller}

Throughout this section, we again consider the situation of Theorem~\ref{t.limit-semigroup}. We seek to impose additional assumptions that ensure that the limit semigroup enjoys the strong Feller property if the same is true for the approximating semigroups $T_n$. 
In contrast to the previous section, we now require the involved semigroups to be sub Markovian.
More precisely, we will assume the following:

\begin{hyp}
	\label{h.sf}
	For each $n\in \N$ let $(T_n(t))_{t>0}$ be a semigroup of sub Markovian kernel operators with pseudo-resolvent $(R_n(\lambda))_{\Re\lambda >0}$. 
	Assume moreover that one of the following conditions is satisfied:
	\begin{enumerate}[(1)]
		\item 
		$R_n(\lambda) \leq R_{n+1}(\lambda)$ for all $\lambda >0$ and $n\in \N$ or
		
		\item 
		$R_n(\lambda) \geq R_{n+1}(\lambda)$ for all $\lambda >0$ and $n\in \N$.
	\end{enumerate}
	In either case, by Theorem \ref{t.limit-semigroup} the sequence if semigroups converges to a sub Markovian semigroup of kernels operators which we denote by $(T(t))_{t>0}$; 
	we denote its pseudo-resolvent by $(R(\lambda))_{\Re\lambda >0}$. 
	Moreover, we assume that the following two conditions are satisfied:
	\begin{enumerate}[(A)]
		\item For every $n\in \N$ the semigroup $T_n$ enjoys the strong Feller property.
		
		\item For some $\lambda >0$, the function $R(\lambda)\one$ is continuous, i.e., an element of $C_b(E)$.
	\end{enumerate}
	Here, $\one$ denotes the constant function on $E$ with value $1$.
\end{hyp}

We will prove in Theorem~\ref{t.strongfeller} that the additional assumptions in Hypothesis \ref{h.sf}(A) and \ref{h.sf}(B) imply that the limit semigroup $(T(t))_{t > 0}$ also enjoys the strong Feller property. 
As a preliminary result, we establish some properties of the set $\mathscr{J}_x(t)$, defined $t\in (0,\infty)$ and $x\in E$ by
\begin{equation}\label{eq.jx}
	\mathscr{J}_x(t) \coloneqq \{ f\in B_b(E) : T(t)f \mbox{ is continuous at the point } x\}.
\end{equation}

\begin{lem}\label{l.jx}
	Let all assumptions in Hypothesis~\ref{h.sf} be satisfied (except for possibly~{\upshape(B)}, which is not needed for this lemma).
	For each $t\in (0,\infty)$ and each $x\in E$ the set $\mathscr{J}_x(t)$ defined in~\eqref{eq.jx} has the following properties: 
	\begin{enumerate}[\upshape (i)]
		\item $\mathscr{J}_x(t)$ is a vector space;
		\item If two vectors $f,g \in B_b(E)$ satisfy $0\leq f \leq g$ and $g \in \mathscr{J}_x(t)$, then $f \in \mathscr{J}_x(t)$;
		\item If $\one\in\mathscr{J}_x(t)$, then $\mathscr{J}_x(t) = B_b(E)$.
	\end{enumerate}
\end{lem}

\begin{proof}
	We give the proof for an increasing sequence of semigroups. In the decreasing case, the proof is -- mutatis mutandis -- analogous.
	
	(i) This is obvious. 
	
	(ii) Let $0\leq f \leq g \in \mathscr{J}_x(t)$. 
	By Hypothesis \ref{h.sf}(A) the functions $T_n(t)f$ are continuous, so $T(t)f = \sup_n T_n(t)f$ is lower semicontinuous as pointwise supremum of continuous functions. Thus, if $x_n \to x$ then $\liminf_{n \to \infty} T(t)f(x_n) \geq T(t)f(x)$. As $f\leq g$, we have $g-f\geq 0$ and it follows that
$T(t)(g-f) = \sup_n T_n(t)(g-f)$. Thus, also $T(t)(g-f)$ is lower semicontinuous. 
	If we again consider a sequence $x_n \to x$, then
	\[ 
		\liminf_{n \to \infty} \bigl [T(t)(g-f)\bigr ](x_n) \geq \bigl [T(t)(g-f)\bigr ](x).
	\]
	However, as $g\in \mathscr{J}_x(t)$, we have $T(t)g(x_n) \to T(t)g(x)$ and $\limsup_{n \to \infty} f(x_n) \leq f(x)$ follows. 
	Altogether, $f(x_n) \to f(x)$ and thus $f\in \mathscr{J}_x(t)$.

	(iii) This is immediate from (i) and (ii).
\end{proof}

Now we can proceed to our main result.

\begin{thm}
	\label{t.strongfeller}
	Let Hypothesis~\ref{h.sf} be satisfied. 
	Then $(T(t))_{t > 0}$ enjoys the strong Feller property.
\end{thm}

\begin{proof}
	We only prove the theorem for a monotonically increasing sequence $(T_n)_{n\in \N}$; 
	the case of a decreasing sequence is similar. 
	Let us begin with some preliminary observations. 
	As $(T(t))_{t > 0}$ is sub Markovian, the orbit $t\mapsto T(t)\one$ is monotonically decreasing. 
	This has two important consequences: 
	
	(i) If we have $\one\in \mathscr{J}_x(t)$ for some $x \in E$ and $t > 0$ then, by Lemma~\ref{l.jx}, we also have $T(s)\one\in \mathscr{J}_x(t)$ for all $s>0$; by the semigroup law, this implies $\one \in \mathscr{J}_x(t+s)$. 
	Thus, by contraposition, if $\one\not\in \mathscr{J}_x(t)$, then $\one\not\in \mathscr{J}_x(s)$ for all $s\in (0,t)$.
	
	(ii) For fixed $x\in E$ the scalar function $t\mapsto T(t)\one(x)$ is also monotonically decreasing, so it has at most countably many discontinuities.\smallskip
	
	To prove the theorem, by Lemma \ref{l.jx} it suffices to prove that $\one \in \mathscr{J}_x(t)$ for all $x\in E$ and $t\in (0,\infty)$. 
	Aiming for a contradiction, let us assume that $\one\not\in \mathscr{J}_{x_0}(t_0)$ for a point $x_0 \in E$ and a time $t_0 > 0$. 
	It follows from points~(i) and~(ii) above that, by making $t_0$ smaller if necessary, we can achieve that $t\mapsto T(t)\one (x_0)$ is continuous at $t_0$.
	As $T(t_0)\one$ is the pointwise supremum of the continuous functions $T_n(t_0)\one$, where $n$ runs through $\N$, it follows that $T(t_0)\one$ is lower semicontinuous. 
	Thus, since $\one\not\in \mathscr{J}_{x_0}(t_0)$ we find a number $\eps>0$ and a sequence $x_n\to x_0$ with	$T(t_0)\one(x_n) \geq T(t_0)\one (x_0) + 2\eps$ for all $n\in \N$.
	
	By monotonicity with respect to time, $T(s)\one (x_n) \geq T(t_0)\one (x_n) \geq T(t_0)\one (x_0) + 2\eps$ for all $s\in (0, t_0)$ and all $n \in \bbN$.
	Since $t\mapsto T(t)\one (x_0)$ is continuous at $t_0$, we can pick
	$\delta >0$ such that $t_0-\delta >0$ and $T(t)\one (x_0) \leq T(t_0)\one (x_0) + \eps$ for all $t\in [t_0-\delta, t_0]$. 
	It follows that $T(t)\one (x_n) \geq T(t)\one (x_0) + \eps$ for all $t\in [t_0-\delta, t_0]$ and all $n \in \bbN$. 
	So to sum up we have
	\begin{align*}
		\liminf_{n \to \infty} T(t)\one (x_n) \geq 
		\begin{cases}
			T(t)\one (x_0) + \eps \quad & \text{for } t \in [t_0-\delta, t_0], \\
			T(t)\one (x_0) \quad & \text{for all other } t > 0,
		\end{cases}
	\end{align*}
	where the inequality in the second case follows from the lower semicontinuity of $T(t)\one$.
	
	Now consider $\lambda$ as in Hypothesis~\ref{h.sf}(B). 
	Using Fatou's Lemma, we find
	\begin{align*}
		\liminf_{n \to \infty} R(\lambda)\one (x_n) & \geq  \int_0^\infty e^{-\lambda t} \liminf_{n \to \infty} T(t)\one (x_n)\, \ud t\\
		& \geq \int_0^\infty e^{-\lambda t} T(t)\one (x_0)\, dt + \int_{t_0-\delta}^{t_0} e^{-\lambda t}\eps\, \ud t\\
		& > R(\lambda) \one (x_0).
	\end{align*}
	This contradicts the assumed continuity of $R(\lambda)\one$.
\end{proof}

\begin{rem}
	Our proof of Theorem~\ref{t.strongfeller} was inspired by the argument presented for the implication `(vii) $\Rightarrow$ (ii)' in \cite[Thm. 3.2]{sch98}.
	We note, however, that this implication in \cite[Thm. 3.2]{sch98} is not correct in the form stated there;
	this was kindly confirmed to us by the author of \cite{sch98}.
	
	Let us briefly provide a counterexample on $E = (0,\infty)$. 
	This can easily be transferred to $\R$ (and thus fits into the framework of \cite{sch98}) via any homeomorphism $(0,\infty) \to \R$.
	For $x,t > 0$, let
	\[
		T(t)f(x) = 
		\begin{cases}
			f(x-t) & \mbox{ for } x> t\\
			0 & \mbox{ for } x\leq t.
		\end{cases}
	\]
	Then $T(t)C_b(E) \not\subseteq C_b(E)$, while the pseudo-resolvent of $T$ is given by
	\[
		R(\lambda)f(x) = e^{-\lambda x} \int_0^x e^{\lambda s} f(s)\, \ud s
	\]
	for $x,\re \lambda > 0$, and thus leaves $C_b(E)$ invariant.
	Note that the semigroup $T$ leaves $C_0(E)$  invariant, and thus satisfies all requirements of the semigroups considered in~\cite{sch98}.
	
	Note that a closely related semigroup appears as a limit semigroup in our Example~\ref{ex.cbnotgood} below.
\end{rem}

In the situation of Theorem \ref{t.strongfeller}, it follows from Lemma \ref{l.sfcontinuous} that for the limit semigroup $T$ the map 
$(t,x) \mapsto T(t)f(x)$ is continuous on $(0,\infty)\times \Omega$ for every $f\in B_b(E)$, 
whenever $R(\lambda)$ acts injectively on $C_b(E)$.
We now address the question of continuity at time $t=0$. 
To this end, we make the following assumption:

\begin{hyp}\label{h.exhaust}
	There exists a sequence $(K_n)_{\in \N}$ of compact sets such that 
	\[\bigcup_{n\in \N} K_n = E\] 
	and $T_n(t)\one (x) \to 1$ as $t\to 0$ for every $n \in \N$ and every $x\in K_n$.
\end{hyp}

Our first result concerning continuity at $0$ does not require the strong Feller property:

\begin{prop}\label{p.contzero}
	Assume that we are in the situation of Theorem~\ref{t.limit-semigroup} for $M=1$ (i.e.,\ all semigroups $T_n$ are sub Markovian) and that the equivalent assertions {\upshape(i)} and {\upshape(ii)} in that theorem are satisfied for the case of increasing semigroups and pseudo-resolvents. 
	Moreover, let Hypothesis~\ref{h.exhaust} be satisfied.
	Then, for all $f\in C_b(E)$, we have $T(t)f \to f$ uniformly on compact sets as $t\to 0$. 
\end{prop} 

\begin{proof}
	We follow the lines of \cite[Lem.\ 3.1]{sch98}. 
	With similar arguments as in Lemma \ref{l.jx}, we see that it suffices to prove 
	that $T(t)\one \to \one$ uniformly on compact sets as $t\to 0$. 
	Actually, by Dini's theorem, it suffices to establish pointwise convergence.
	So fix $x\in E$ and $\eps >0$. 
	By Hypothesis~\ref{h.exhaust}, we find $n_0\in \N$ with $x\in K_{n_0}$. 
	Moreover, $T_{n_0}(t)\one(x) \to 1$ as $t\to 0$, also by Hypothesis~\ref{h.exhaust}. 
	
	Consequently, there is a number $\delta >0$ with $T_{n_0}(t)\one(x) \geq 1-\eps$ for all $t\in [0,\delta]$. 
	It follows that $1\geq T(t)\one(x)\geq T_{n_0}(t)\one(x) \geq 1-\eps$ for all $t\in [0,\delta]$, proving that indeed $T(t)\one(x)\to 1$ as $t \to \infty$.
\end{proof}

We should point out that for \emph{decreasing} sequences, we cannot expect continuity of the limit semigroup at $0$. Indeed, already for 
$E=\{0\}$, i.e.,\ $B_b(E) \simeq \cM_b(E) \simeq \R$, the semigroups $T_n$ given by $T_n(t) = e^{-nt}$ are monotonically decreasing to the zero semigroup $0$ on $(0,\infty)$, which does not converge to $1$ as $t\to 0$.

In the situation of Proposition \ref{p.contzero}, it is natural to extend the semigroup $T$ which, up to now, is only defined on $(0,\infty)$ to
$[0,\infty)$ by setting $T(0) = I$. Then Proposition \ref{p.contzero} states that this extended semigroup is $\sigma$-continuous at $0$. In the following result, we also set $T_n(0)=I$ for the approximative semigroups $T_n$. We note that, in general, this is \emph{not} $\sigma$-continuous 
at $0$ (as typically $T_n(t)f(x) \equiv 0$ for $x\in E\setminus K_n$). However, repeating the arguments from Lemma \ref{eq.jx}, we see that Hypothesis \ref{h.exhaust} yields $T_n(t)f(x) \to f(x)$ as $t\to 0$ for every $x\in K_n$.

\begin{cor}\label{c.cz1}
	Assume that Hypothesis~\ref{h.sf} is satisfied with option {\upshape(1)} and that Hypothesis~\ref{h.exhaust} is satisfied.
	Then $T$ is a $C_b$-Feller semigroup and for every $f\in C_b(E)$ we have $T_n(t)f (x) \to T(t)f(x)$, 
	uniformly for $(t,x)$ in compact subsets of $[0,\infty)\times E$. 
	Moreover, the $C_b$-generator $A_{C_b}$ of $T$ can be characterized in terms of the full generators $A_n$ of $T_n$ as follows: 
	for all $u,f \in C_b(E)$ we have
	\begin{align*}
		& u \in D(A_{C_b}) \quad \text{and} \quad A_{C_b}u=f \\
		\Leftrightarrow \qquad
		& \exists (u_n,f_n) \in A_n\cap C_b(E)\times C_b(E) : u_n \weak u, f_n \weak f.
	\end{align*}
	In the second line, we can equivalently require that the sequences $u_n$ and $f_n$ be uniformly bounded 
	and converge uniformly on compact sets to $u$ and $f$, respectively.
\end{cor}

\begin{proof}
	By Theorem \ref{t.strongfeller}, $T$ enjoys the strong Feller property (thus, in particular, $C_b(E)$ is invariant under $T$), 
	and Proposition \ref{p.contzero} yields stochastic continuity of $T$, so $T$ is a $C_b$-Feller semigroup. 
	In view of \cite[Thm.\ 2.10]{k09}, the stochastic continuity implies that the operators in the pseudo-resolvent $(R(\lambda))_{\Re\lambda >0}$ are injective on $C_b(E)$.
	Thus, by Lemma~\ref{l.sfcontinuous} and Proposition~\ref{p.contzero}, for $f\in C_b(E)$ the map $(t,x)\mapsto T(t)f(x)$ is continuous on $[0,\infty)\times E$. 
	The claimed uniform convergence on compact sets of the semigroups follows for positive functions from Dini's theorem. 
	The general case follows by splitting a general function into positive and negative part.
	
	As $A_n\cap C_b(E)\times C_b(E)$ is the generator of the restricted semigroup $T_n|_{C_b(E)}$,
	the convergence of the generators can be established as in the proof of Corollary~\ref{c.generatorconv}. 
	For the addendum we note that, by Dini's theorem, 
	we actually have $R_n(1)f \to R(1)f$ uniformly on compact sets for every $f\in C_b(E)$ (even $f\in B_b(E)$).
\end{proof}

We close this section we a note on the behaviour of the limit semigroup on the space $C_0(E)$ which denotes, 
in case that the Polish space $E$ is locally compact, the space of continuous functions on $E$ that vanish at infinity.

\begin{cor}\label{c.cz2}
	Suppose that the Polish space $E$ is locally compact.
	Assume that Hypothesis~\ref{h.sf} is satisfied with option {\upshape(1)}, 
	that Hypothesis~\ref{h.exhaust} is satisfied and that $T(t)C_0(E)\subset C_0(E)$ for all $t>0$. 
	Then $T(t)f \to f$ as $t \to 0$ for every $f\in C_0(E)$, 
	and we have $T_n(t)f\to T(t)f$ with respect to $\|\cdot\|_\infty$, uniformly for $t$ in compact subsets of $[0,\infty)$.
\end{cor}

\begin{proof}
	If we set $T(0) = I$, then the semigroup $(T(t)|_{C_0(E)})_{t \in [0,\infty)}$ is, by Corollary~\ref{c.cz1}, weakly continuous; 
	hence it is strongly continuous (see \cite[Thm.\ I.5.8]{en}). 
	
	Next, fix $T>0$ and $f\in C_0(E)$. 
	Given $\eps>0$, by the strong continuity and the compactness of $[0,T]$, 
	we find a compact set $K$ such that $|T(t)f(x)|\leq \eps$ (and thus $|T(t)f(x) - T_n(t)f(x)| \leq 2\eps$) 
	for all $t\in [0,T]$, $n \in \N$ and $x\in E\setminus K$. 
	By Corollary \ref{c.cz1}, there is $n_0$ with $|T(t)f(x)-T_n(t)f(x)|\leq 2\eps$ for all $n\geq n_0$, $t\in [0,T]$ and $x\in K$ 
	so that $\|T(t)f-T_n(t)f\|_\infty\leq 2\eps$ for all $n\geq n_0$ and $t\in [0,T]$. 
\end{proof}

\section{Counterexamples}
\label{sec.counterexamples}

Clearly, Hypothesis~\ref{h.sf}(A) cannot be omitted in Theorem~\ref{t.strongfeller}.
As a simple counterexample, one can take $E$ to be the complex unit circle and $T_n = T$ to be the shift semigroup on $B_b(E)$ for each $n$. 
Then the pseudo resolvent of $T$ does not only map $\one$ to a constant function, it is even strongly Feller. 
Still, the semigroup $T$ is not strongly Feller.

The following example shows a slightly more involved phenomenon: here, the pseudo resolvent of the limit semigroup is also strongly Feller, but the limit semigroup does not even leave $C_b(E)$ invariant (in contrast to the approximating semigroups).

\begin{example}\label{ex.cbnotgood}
	Let $E= (0,1]$ and define for each $n\in \N$ the semigroup $(T_n(t))_{t > 0}$ on $B_b(E)$ by setting
	\[
		T_n(t)f(x) = 
		\begin{cases}
			\Big(\frac{x-t}{x}\Big)^{\frac{1}{n}}f(x-t) & \mbox{ for } x> t\\
			0 & \mbox{ for } x\leq t.
		\end{cases}
	\]
	It is easy to see that $(T_n(t))_{t > 0}$ is a sub Markovian semigroup of kernel operators. Moreover, $T_n(t)C_b(E) \subseteq C_b(E)$ (actually, $T_n(t)C_b(E) \subseteq C_0(E)$) but $(T_n(t))_{t > 0}$ does not enjoy the strong Feller property. Note that for fixed $x>t$ we have $(x-t)/x \in (0,1)$ and hence $\sqrt[n]{(x-t)/x} \uparrow 1$ as $n \to \infty$. 
	Consequently, the semigroups $T_n$ are monotonically increasing to the semigroup $T$ given by
	\[
		T(t)f(x) = 
		\begin{cases}
			f(x-t) & \mbox{ for } x> t\\
			0 & \mbox{ for } x\leq t.
		\end{cases}
	\]
	Note that $T(t)C_b(E) \not\subseteq C_b(E)$ even though the pseudo-resolvent $(R(\lambda))_{\Re \lambda >0}$ of $T$ is given by
	\[
		R(\lambda)f(x) = e^{-\lambda x} \int_0^x e^{\lambda s} f(s)\, \ud s,
	\]
	and thus satisfies $R(\lambda) B_b(E)\subseteq C_b(E)$ (even $R(\lambda) B_b(E)\subseteq C_b(E)$) for all $\Re\lambda >0$, which means that Hypothesis~\ref{h.sf}(B) is satisfied.
\end{example}

We point out that if the limit semigroup $T$ enjoys the strong Feller property (or if merely $T(t)\one$ is continuous for every $t>0$)
then Hypothesis \ref{h.sf}(B) is necessarily satisfied. Our second example shows that  Hypothesis~\ref{h.sf}(B) cannot be dropped 
in Theorem~\ref{t.strongfeller}.

\begin{example} \label{ex.loose-strong-feller}
	Let $d\geq 3$ and put $\Omega_n\coloneqq \R^d\setminus B(0, 1/n)$. 
	We let $T_n$ be the semigroup generated by the Dirichlet Laplacian on $\Omega_n$ and $R_n = (R_n(\lambda))_{\re \lambda > 0}$ be the Laplace transform of that semigroup. 
	Here, we consider functions on $\Omega_n$ as functions on all of $\R^d$, extending them by zero outside of $\R^d$. 
	By the Dirichlet boundary conditions and the regularity of $\Omega_n$ 
	it follows that the semigroup $T_n$ and the pseudo resolvent $R_n$ consist of strong Feller operators. 
	As a consequence of the maximum principle, 
	the resolvents $R_n(\lambda)$ (for $\lambda >0$) and the semigroups $T_n$ are monotonically increasing in $n$. 
	We write $T(t)f = \sigma\text{-}\lim_n T_n(t)f$ for $t > 0$ and for $f\in B_b(\R^d)$;
	it follows from Theorem~\ref{t.limit-semigroup} that $T$ is a sub Markovian semigroup of kernel operators on $B_b(\R^d)$
	and that its pseudo resolvent $R$ satisfies $R(\lambda)f = \sigma\text{-}\lim_n R_n(\lambda) f$ whenever $f \in B_b(\R^d)$ and $\re \lambda > 0$.
	
	Now, fix $\lambda > 0$.
	By \cite[Prop.\ 3.3]{ad07}, $R_n(\lambda)f$ converges uniformly on compact subsets of $\R^d\setminus\{0\}$ to the resolvent of the Dirichlet Laplacian on $\R^d\setminus\{0\}$ applied to $f$ whenever $f \in B_b(\R^d)$. 
	Moreover, as $d\geq 3$, singletons have zero capacity whence $\Delta_{\R^d} = \Delta_{\R^d\setminus \{0\}}$, cf.\ \cite[Cor.\ 3.9]{a02}. 
	So it follows that $R(\lambda)\one(x) = 1/\lambda$ for all $x \in \R^d \setminus \{0\}$.
	
	However, we have $R_n(\lambda)\one(0) = 0$ for all $n$, so that $R(\lambda)f(0) = 0$. 
	We thus conclude that $R(\lambda)\one = \lambda^{-1}\one_{\R^d\setminus\{0\}}$ is not continuous in $0$,
	so Hypothesis~\ref{h.sf}(B) is not satisfied.
	
	As for the semigroup $T$, it is clear that $T(t)\one (0) = 0$ for all $t>0$. 
	For points $x\neq 0$, the uniqueness theorem for Laplace transforms shows that $T(t)\one (x) = 1$ for almost all $t>0$.
	But since $t\mapsto T(t)\one (x)$ is monotonically decreasing, 
	we must have $T(t)\one(x) = 1$ for all $x\neq 0$ and $t > 0$.
	This proves that $T$ does not enjoy the strong Feller property.
\end{example}

\section{Application: Elliptic operators with unbounded coefficients}
\label{sec.application}

In this section, we prove that a certain realization of an elliptic operator with possibly unbounded coefficients is the full generator of a sub Markovian semigroup of kernel operators that enjoys the strong Feller property. Our strategy is similar to that employed in \cite{mpw02}. 
However, making use of our theoretical results, we obtain two major benefits compared to~\cite{mpw02}: 
On the one hand, the proof simplifies significantly as it now suffices to study the resolvent equation. 
On the other hand, we do not need to impose additional regularity assumptions on the coefficients to construct the semigroup.

We will consider differential operators $\cA$ of the form
\begin{equation}
	\label{eq.operator}
	\cA u = \sum_{i,j=1}^da_{ij}\frac{\partial^2}{\partial x_i\partial x_j}u + \sum_{j=1}^d b_j\frac{\partial}{\partial x_j}u.
\end{equation}
Throughout, we make the following assumptions on the coefficients.

\begin{hyp}
	\label{h.1}
	For $i,j=1, \ldots, d$ let $a_{ij}: \R^d \to \R$ be continuous and $b_j : \R^d \to \R$ be measurable functions 
	such that the following properties hold:
	\begin{enumerate}
		\item 
		We have $a_{ij}=a_{ji}$, and 
		\[
			\sum_{i,j=1}^d a_{ij}(x)\xi_i\xi_j \geq \eta(x) |\xi|^2 \quad \text{for all } x,\xi\in\R^d,
		\]
		where $\eta:\R^d\to (0,\infty)$ is a function that satisfies $\inf_K\eta >0$ for all $K\Subset \R^d$.
		
		\item 
		The functions $b_j$ are locally bounded.
	\end{enumerate}
\end{hyp}

To prove that a certain realization of the operator $\cA$ generates a semigroup of kernel operators we use, similar to \cite{mpw02}, an approximation argument. 
To that end, we define the multivalued operator $A_n$ on $B_b(B(0,n))$ by setting
\begin{equation}
	\label{eq.approxop}
	(u, f)\in A_n \Leftrightarrow \Big[ u \in C_0(B(0,n)) \cap \bigcap_{1<p<\infty}W^{2,p}(B(0,n)) \mbox{ and } \cA u = f \mbox{ a.e.\ on } \Omega\Big]
\end{equation}
for $u,f \in B_b(B(0,n))$.
Note that for $u\in \bigcap_{p\in (1,\infty)}W^{2,p}(B(0,n))$ the function $\cA u$ is only well-defined modulo equality almost everywhere. Thus, we can rephrase the above by saying that $(u,f)\in A_n$ if and only if $f$ is a version of $\cA u$. This makes $A_n$ indeed a multivalued operator. Alternatively, we could consider a single-valued operator $\tilde A_n$ on $L^\infty(B(0,n))$ by setting $D(\tilde A_n) = C(\bar{B}(0,n))\cap\bigcap_{p\in (1,\infty)} W^{2,p}{B(0,n)}$ and $\tilde A_n u = [\cA u]$ where the square brackets refer to the equivalence class in $L^\infty(B(0,n))$.

\begin{lem}\label{l.approxgen}
	The operator $A_n$ is the full generator of a sub Markovian semigroup $T_n = (T_n(t))_{t>0}$ of kernel operators on $B_b(B(0,n))$. Moreover, $T_n$ is a $C_b$-Feller semigroup that enjoys the strong Feller property.
\end{lem}

\begin{proof}
	That $A_n$ is a generator follows from \cite[Cor.\ 3.1.21]{l13} (cf.\ also \cite{as14} and \cite[Thm.\ 3.5]{akk16}). Note that, actually,  in these references generation of a semigroup $\tilde T_n$ on the space $L^\infty(B(0,n))$ by the operator $\tilde A_n$ introduced above was established. 
	However, if $q : B_b(B(0,n)) \to L^\infty(B(0,n))$ is defined by $q f =[f]$, it is easy to see that $T_n = \tilde T_n \circ q$ is a semigroup on $B_b(B(0,n))$ with full generator $A_n = \{ (u,f) : u \in D(\tilde A), q f =\tilde A u \}$. That this semigroup $T_n$ consists of kernel operators which enjoy the strong Feller property follows from \cite[Prop.\ 5.7]{akk16} with $\mu\equiv 0$. As for the stochastic continuity, first note that in view of \cite[Prop.\ 2.1.4]{l13}, it is a further consequence of \cite[Cor.\ 3.1.21]{l13} that $T_n(t)f \to f$ with respect to $\|\cdot\|_\infty$ whenever $f\in C_0(B(0,n))$. By \cite[Lem.\ 3.1]{sch98}, it follows that $T_n$ is stochastically continuous.
\end{proof}

We now turn to the elliptic operator on all of $\R^d$. We set
\[
	D_{\mathrm{max}}(A) \coloneqq \Big\{ u\in C_b(\R^d) \cap \bigcap_{1<p<\infty} W^{2,p}_{\mathrm{loc}}(\R^d) : \exists f\in B_b(\R^d) \mbox{ s.t. } \cA u = f \mbox{ a.e.}\Big\}.
\]

\begin{thm}
	\label{t.ellipticmain}
	There exists a subspace $\hat D \subset D_{\mathrm{max}} (A)$ such that the operator $\hat A \coloneqq \{ (u, f) \in \hat D\times B_b(\R^d) : \cA u = f \mbox{ a.e.}\}$ is the full generator of a sub Markovian semigroup $T=(T(t))_{t>0}$ of kernel operators. 
	This semigroup is a $C_b$-Feller semigroup and enjoys the strong Feller property. 
\end{thm}

\begin{proof}
	Given $n\in \N$, let $A_n$ be as in Lemma \ref{l.approxgen} and  consider the resolvent $R_n(\lambda)\coloneqq R(\lambda, A_n)$ and the semigroup $T_n$,  
	generated by $A_n$ as operators on all of $\R^d$, extending functions by zero outside $B(0,n)$.
	Note that in view of the imposed Dirichlet boundary conditions, this results in strong Feller operators.

	Making use of the maximum principle, see \cite[Lem.\ 3.2]{as14} for a version suitable for our setting, 
	it is easy to see that for $0 \le f\in C_b(\R^d)$ and $\lambda >0$ the sequence $(R_n(\lambda)f)$ is monotonically increasing. 
	This implies that the same monotonicity remains true when the operators act on $B_b(\R^d)$. 
	Now Theorem~\ref{t.limit-semigroup} shows that the semigroups $T_n$ are monotonically increasing to a semigroup $T$ whose Laplace transform is $R(\lambda) \coloneqq \sup_n R_n(\lambda)$. 
	
	To employ Theorem \ref{t.strongfeller}, we still need to verify Assumption (B) in Hypothesis \ref{h.sf}. To that end, put $u=R(\lambda)\one$ and $u_n= R_n(\lambda)\one$. By construction, $u_n$ converges pointwise to $u$. Fix a bounded open set $U\subset \R^d$ and pick $n_0 \in \N$ so large, that $\bar U \subset B(0, n_0)$. By \cite[Lem.\ 3.4]{akk16}, there is a constant $C=C(U, n_0)$ such that
	\[
	\|u_n\|_{C^1(U)} \leq C \|\one\|_{L^\infty(B(0, n_0))}
	\]
for all $n\geq n_0$. By the Arzel\`a--Ascoli Theorem, $u_n$ has a subsequence that converges, uniformly on $\bar U$, to a continuous function. This proves that $u$ is continuous on $\bar U$ and thus, as we may exhaust $\R^d$ with bounded open sets, continuous on all of
$\R^d$. Now Theorem~\ref{t.strongfeller} yields  the strong Feller property for $T$. 
As Hypothsis \ref{h.exhaust} is satisfied, Corollary \ref{c.cz1} shows that $T$ is a $C_b$-Feller semigroup.

	It remains to identify the generator of $T$:
	repeating the arguments in the proof of \cite[Thm.\ 3.4]{mpw02}, we see that for $0 \le f \in B_b(\R^d)$ the sequence $u_n \coloneqq R_n(\lambda)f$ converges, uniformly on compact sets, to a function $u\in D_\mathrm{max}(A)$ that solves the elliptic equation $\lambda u - \cA u = f$.
	As $R(\lambda)$ is a pseudoresolvent, it follows that  there exists an operator $\hat A$ as claimed.
\end{proof}

\begin{rem}
	One might ask for the long-term behaviour of the semigroup constructed in Theorem~\ref{t.ellipticmain}.
	In \cite[Sect.\ 7.1]{ggk22} this was studied under the stronger regularity assumptions on the coefficients made in \cite{mpw02}.
	Now that the semigroup can also be obtained under weaker assumptions, similar methods as in \cite[Sect.\ 7.1]{ggk22} can be applied to this semigroup, too.
\end{rem}

\subsection*{Acknowledgements}

We are indebted to René Schilling for a useful discussion concerning Theorem~3.2 in his article \cite{sch98}.

C.B.\ acknowledges funding by the Deutsche Forschungsgemeinschaft (DFG, German Research Foundation) -- 468736785.


\begin{thebibliography}{10}

\bibitem{a02}
{\sc W.~Arendt}, {\em Does diffusion determine the body?}, J. Reine Angew.
  Math., 550 (2002), pp.~97--123.

\bibitem{abhn2011}
{\sc W.~Arendt, C.~J.~K. Batty, M.~Hieber, and F.~Neubrander}, {\em
  Vector-valued {L}aplace transforms and {C}auchy problems}, vol.~96 of
  Monographs in Mathematics, Birkh\"{a}user/Springer Basel AG, Basel,
  second~ed., 2011.

\bibitem{ad07}
{\sc W.~Arendt and D.~Daners}, {\em Uniform convergence for elliptic problems
  on varying domains}, Math. Nachr., 280 (2007), pp.~28--49.

\bibitem{akk16}
{\sc W.~Arendt, S.~Kunkel, and M.~Kunze}, {\em Diffusion with nonlocal boundary
  conditions}, Journal of Functional Analysis, 270 (2016), pp.~2483 -- 2507.

\bibitem{as14}
{\sc W.~Arendt and R.~M. Sch\"{a}tzle}, {\em Semigroups generated by elliptic
  operators in non-divergence form on {$C_0(\Omega)$}}, Ann. Sc. Norm. Super.
  Pisa Cl. Sci. (5), 13 (2014), pp.~417--434.

\bibitem{b21}
{\sc J.~{Banasiak}}, {\em {How to be positive in natural sciences?}}, in
  Positivity and its applications. Proceedings from the conference Positivity
  X, Pretoria, South Africa, July 8--12, 2019, Cham: Birkh\"auser, 2021,
  pp.~31--62.

\bibitem{bl17}
{\sc J.~{Banasiak} and M.~{Lachowicz}}, {\em {Around the Kato generation
  theorem for semigroups}}, {Stud. Math.}, 179 (2007), pp.~217--238.

\bibitem{bte14}
{\sc C.~J.~K. Batty and A.~F.~M. ter Elst}, {\em On series of sectorial forms},
  J. Evol. Equ., 14 (2014), pp.~29--47.

\bibitem{bobrowski2016}
{\sc A.~Bobrowski}, {\em Convergence of one-parameter operator semigroups},
  vol.~30 of New Mathematical Monographs, Cambridge University Press,
  Cambridge, 2016.
\newblock In models of mathematical biology and elsewhere.

\bibitem{cte18}
{\sc R.~Chill and A.~F.~M. ter Elst}, {\em Weak and strong approximation of
  semigroups on {H}ilbert spaces}, Integral Equations Operator Theory, 90
  (2018), pp.~Paper No. 9, 22.

\bibitem{es10}
{\sc T.~Eisner and A.~Ser\'{e}ny}, {\em On the weak analogue of the
  {T}rotter-{K}ato theorem}, Taiwanese J. Math., 14 (2010), pp.~1411--1416.

\bibitem{en}
{\sc K.-J. Engel and R.~Nagel}, {\em One-parameter semigroups for linear
  evolution equations}, vol.~194 of Graduate Texts in Mathematics,
  Springer-Verlag, New York, 2000.
\newblock With contributions by S. Brendle, M. Campiti, T. Hahn, G. Metafune,
  G. Nickel, D. Pallara, C. Perazzoli, A. Rhandi, S. Romanelli and R.
  Schnaubelt.

\bibitem{ek}
{\sc S.~N. Ethier and T.~G. Kurtz}, {\em Markov processes}, Wiley Series in
  Probability and Mathematical Statistics: Probability and Mathematical
  Statistics, John Wiley \& Sons, Inc., New York, 1986.
\newblock Characterization and convergence.

\bibitem{ggk22}
{\sc M.~Gerlach, J.~Gl\"uck, and M.~Kunze}, {\em Stability of transition
  semigroups and applications to parabolic equations}, to appear in Trans.
  Amer. Math. Soc.

\bibitem{gk15}
{\sc M.~Gerlach and M.~Kunze}, {\em On the lattice structure of kernel
  operators}, Math. Nachr., 288 (2015), pp.~584--592.

\bibitem{haase}
{\sc M.~Haase}, {\em The functional calculus for sectorial operators}, no.~v.
  169 in Operator theory, advances and applications, Birkh\"auser Verlag, 2006.

\bibitem{k54}
{\sc T.~{Kato}}, {\em {On the semi-group generated by Kolmogoroff's
  differential equations}}, {J. Math. Soc. Japan}, 6 (1954), pp.~1--15.

\bibitem{kr09}
{\sc S.~Kr\'{o}l}, {\em A note on approximation of semigroups of contractions
  on {H}ilbert spaces}, Semigroup Forum, 79 (2009), pp.~369--376.

\bibitem{k09}
{\sc M.~Kunze}, {\em Continuity and equicontinuity of semigroups on norming
  dual pairs}, Semigroup Forum, 79 (2009), pp.~540--560.

\bibitem{k11}
\leavevmode\vrule height 2pt depth -1.6pt width 23pt, {\em A {P}ettis-type
  integral and applications to transition semigroups}, Czechoslovak Math. J.,
  61(136) (2011), pp.~437--459.

\bibitem{k20}
\leavevmode\vrule height 2pt depth -1.6pt width 23pt, {\em Diffusion with
  nonlocal {D}irichlet boundary conditions on unbounded domains}, Studia Math.,
  253 (2020), pp.~1--38.

\bibitem{lorenzi2017}
{\sc L.~Lorenzi}, {\em Analytical methods for {K}olmogorov equations},
  Monographs and Research Notes in Mathematics, CRC Press, Boca Raton, FL,
  2017.

\bibitem{l13}
{\sc A.~Lunardi}, {\em Analytic semigroups and optimal regularity in parabolic
  problems}, Modern Birkh\"{a}user Classics, Birkh\"{a}user/Springer Basel AG,
  Basel, 1995.
\newblock [2013 reprint of the 1995 original].

\bibitem{mpw02}
{\sc G.~Metafune, D.~Pallara, and M.~Wacker}, {\em Feller semigroups on
  {$\mathbf{R}^N$}}, Semigroup Forum, 65 (2002), pp.~159--205.

\bibitem{o95}
{\sc E.-M. Ouhabaz}, {\em Second order elliptic operators with essential
  spectrum {$[0,\infty)$} on {$L^p$}}, Comm. Partial Differential Equations, 20
  (1995), pp.~763--773.

\bibitem{revuz}
{\sc D.~Revuz}, {\em Markov chains}, North-Holland Publishing Co., Amsterdam,
  1975.
\newblock North-Holland Mathematical Library, Vol. 11.

\bibitem{sch98}
{\sc R.~L. Schilling}, {\em Conservativeness and extensions of {F}eller
  semigroups}, Positivity, 2 (1998), pp.~239--256.

\bibitem{simon}
{\sc B.~Simon}, {\em A canonical decomposition for quadratic forms with
  applications to monotone convergence theorems}, J. Functional Analysis, 28
  (1978), pp.~377--385.

\bibitem{tyran2020}
{\sc M.~Tyran-Kami\'{n}ska}, {\em Transport equations and perturbations of
  boundary conditions}, Math. Methods Appl. Sci., 43 (2020), pp.~10511--10531.

\end{thebibliography}
\end{document}